\documentclass[12 pt]{article}
\usepackage{stmaryrd}
\usepackage{times}
\usepackage{booktabs}
\usepackage{subfigure}
\usepackage{longtable}
\usepackage{setspace}
\usepackage{tabularx}
\usepackage{multirow, makecell}
\usepackage{rotating}
\usepackage{pifont}
\usepackage{floatrow}
\usepackage{caption}
\usepackage{mathrsfs}
\usepackage[fleqn]{amsmath}
\usepackage{amsfonts,amsthm,amssymb,mathrsfs,bbding}
\usepackage{txfonts}
\usepackage{graphics,multicol}
\usepackage{graphicx}
\usepackage{color}
\usepackage{caption}
\usepackage{indentfirst}
\usepackage{cite}
\usepackage{latexsym,bm}
\usepackage{enumerate}
\usepackage{epsfig}
\evensidemargin=\oddsidemargin\topmargin=-1.5cm

\newtheorem{theorem}{Theorem}[section]
\newtheorem{lemma}{Lemma}[section]

\theoremstyle{definition}

\addtocounter{section}{0}
\begin{document}
\title{On the third ABC index of trees and unicyclic graphs}
\author{
{\small Rui Song$^{a,b,}$\footnote{Corresponding author.
\newline{\it \hspace*{5mm}Email addresses:} 15819931748@163.com(R. Song).}\ \ , Jianbin Cui$^{a,b}$\ \ , Xiaozhong Qi$^{a}$}\\[2mm]
\footnotesize $^{a}$School of Mathematics and Information Engineering, Longdong University,\\ \footnotesize Qingyang, Gansu 745000, P.R.China\\
\footnotesize $^b$Institute of Applied Mathematics, Longdong University, Qingyang, Gansu 745000, P.R.China
}
\date{}
\maketitle {\flushleft\large\bf Abstract:}
Let $G=(V,E)$ be a simple connected graph with vertex set $V(G)$ and edge set $E(G)$. The third atom-bond connectivity index is defined as $ABC_3(G)=\sum\limits_{uv\in E(G)}\sqrt{\frac{e(u)+e(v)-2}{e(u)e(v)}}$, where eccentricity $e(u)$ is the largest distance between $u$ and any other vertex of $G$, namely $e(u)=\max\{d(u,v)|v\in V(G)\}$. This work determines the maximal $ABC_3$ index of unicyclic graphs with any given girth and trees with any given diameter, and characterizes the corresponding graphs.

\begin{flushleft}

\end{flushleft}
\section{Introduction}
Molecular descriptors are playing significant role in the chemistry, pharmacology and biology. Among them, topological indices have a prominent place \cite{1}. One of the best known and widely used indices is the connectivity index $\chi$ introduced in 1975 by Milan Randi\'{c} \cite{2}.

For any graph $G=(V,E)$ with vertex set $V(G)$ and edge set $E(G)$, let $|V(G)|=n$ and $|E(G)|=m$ and  $d(u)$ denote the degree of the vertex $u\in V(G)$. The distance $d_{G}(u,v)$ between vertices $u$ and $v$ is defined as the length of any shortest path in $G$ connecting $u$ and $v$. For a vertex $u\in V(G)$, its eccentricity $e_{G}(u)$, or $e(u)$ for short, is the largest distance between $u$ and any other vertex of $G$. The diameter $d(G)$ of $G$ is defined as the maximum eccentricity among the vertices of $G$, similarly, the radius $r(G)$ is defined as the minimum eccentricity among the vertices of $G$. There are some topological indices which are related to the eccentricity, for example, the eccentric connectivity index, Zagreb eccentricity indices, eccentric distance sum.

Ernesto Estrada et al. \cite{3} proposed a topological index named atom-bond connectivity index, or $ABC$ index for short, using a modification of the Randi\'{c} connectivity index, as follows.
$$
ABC(G)=\sum\limits_{uv\in E(G)}\sqrt{\frac{d(u)+d(v)-2}{d(u)d(v)}}.
$$
$ABC$ index has been proven to be a valuable predictive index in the study of the heat formation of alkanes \cite{3,4}, some of its mathematical properties are obtained in \cite{5,6,16,17,18}. A. Graovac and M. Ghorbani \cite{7} proposed a new index, nowadays known as the second $ABC$ index, as
$$
ABC_{2}(G)=\sum\limits_{uv\in E(G)}\sqrt{\frac{n(u)+n(v)-2}{n(u)n(v)}},
$$
where $n(u)$ is the number of vertices of $G$ whose distance to the vertex $u$ is smaller than the distance to the vertex $v$. In the chemistry applications, it is used to model both the boiling point and melting point of the molecules. Some papers have given a lot of properties and upper bounds for the second $ABC$ index \cite{8,9,10,11,12,19}.

In 2016, Dae-won Lee \cite{13} defined the third $ABC$ index, $ABC_3$ index for short, of graph $G$ as
$$
ABC_{3}(G)=\sum\limits_{uv\in E(G)}\sqrt{\frac{e(u)+e(v)-2}{e(u)e(v)}},
$$
some lower and upper bounds on third $ABC$ index of graphs were obtained therein.

Let $U_{n}(k)$ be the set of $n$-vertex unicyclic graphs with girth $k$ and $U_{n}$ be the set of all unicyclic grahs of order $n$. In section 2, we investigate the $ABC_{3}$ index of unicyclic graphs in $U_{n}(k)$ and characterize the extremal graphs with the maximal $ABC_{3}$ index in $U_{n}$; in section 3, we determine the second maximal $ABC_{3}$ index among all unicyclic graphs in $U_{n}$. In section 4, we characterize the extremal trees with the maximal $ABC_{3}$ index with any given diameter.

\section{The maximal third $ABC$ index of unicyclic graphs}

If $k=n,n-1$, then $U_{n}(k)$ is unique. So in what follows we assume that $3\leq k\leq n-2$. The following lemma is obviously true, the readers can find its proof in \cite{14}.

\begin{lemma}[\cite{14}]\label{lem-2-1}
If $u,v$ are two adjacent vertices of graph $G$, then $|e_{G}(u)-e_{G}(v)|\leq 1$.
\end{lemma}

For any real numbers $x,y\geq 1$, let $f(x,y)=\sqrt{\frac{x+y-2}{xy}}$. Then  $f(x+1,x)=\sqrt{\frac{2x-1}{x^{2}+x}}$, $f(x,x)=\sqrt{\frac{2x-2}{x^{2}}}$.

\begin{lemma}[\cite{15}]\label{lem-2-2}
$f(x,1)$ is strictly increasing for $x$, $f(x,2)=\sqrt{\frac{1}{2}}$, and $f(x,y)$ is strictly decreasing for $x$ if $y\geq 3$.
\end{lemma}

\begin{lemma}\label{lem-2-3}
If $x\geq 2$, then $f(x+1,x)$ and $f(x,x)$ are strictly decreasing.
\end{lemma}
\begin{proof}
Let $g(x)=\frac{2x-1}{x^2+x}$. Then $\frac{dg(x)}{dx}=\frac{-2x^2+2x+1}{(x^2+x)^2}<0$ since $x\geq 2$, and so $f(x+1,x)$ is strictly decreasing. Similarly, $f(x,x)$ is also strictly decreasing when $x\geq 2$.
\end{proof}

Let $G_{0}$ be a connected graph of order $n$ with at least two vertices and $u_{0}$ be one of its vertex. Let $G_{1}$ be the graph obtained by identifying vertex $u_{0}$ of $G_{0}$ and a pendent vertex of a star $S_{t+2}$, and $G_{2}$ be the graph obtained by identifying vertex $u_{0}$ of $G_{0}$ and the center of a star $S_{t+2}$.

\begin{lemma}\label{lem-2-4}
Let $t\geq2$ be an integer then $ABC_{3}(G_{1})\leq ABC_{3}(G_{2})$, with equality holding if and only if $G_0\cong S_n$ and $u_{0}$ is its center.
\end{lemma}
\begin{proof}
Let $d_0=e_{G_0}(u_0)=\max\{d_{G_0}(u,u_0)|u\in V(G_0)\}$. Then
\begin{eqnarray}
ABC_{3}(G_{1})-ABC_{3}(G_{2})
&=&\sum\limits_{uv\in E(G_{0})}\left(f(e_{G_{1}}(u),e_{G_{1}}(v))-f( e_{G_{2}}(u),e_{G_{2}}(v))\right) \nonumber \\
& &+\sum\limits_{uv\in E(S_{t+2})}\left(f(e_{G_{1}}(u),e_{G_{1}}(v))-f(e_{G_{2}}(u),e_{G_{2}}(v))\right).  \nonumber
\end{eqnarray}
Since $G_0$ is a connected graph with $|G_0|\geq 2$, it follows that $d(G_0)\geq d_0\geq 1$.

\textbf{Case 1.} $d(G_0)=d_0$. If $d(G_0)=d_0=1$, then $G_0$ is a complete graph $K_n$. In this subcase, for each edge $uv$ of $G_1$ with $e_{G_1}(u)=3=e_{G_1}(v)$, we deduce that $uv\in E(G_0)$ and $e_{G_2}(u)=2=e_{G_2}(v)$. Hence,
\begin{eqnarray}
&&ABC_{3}(G_{1})-ABC_{3}(G_{2})\nonumber \\
&=&\sum\limits_{uv\in E(G_{0})}\left(f(e_{G_{1}}(u),e_{G_{1}}(v))-f(e_{G_{2}}(u),e_{G_{2}}(v))\right) \nonumber \\
& &+\sum\limits_{uv\in E(S_{t+2})}\left(f(e_{G_{1}}(u),e_{G_{1}}(v))-f(e_{G_{2}}(u),e_{G_{2}}(v))\right) \nonumber \\
&=&(\frac{n(n-1)}{2}-(n-1))(f(3,3)-f(2,2))+(n-1)(f(3,2)-f(2,1))\nonumber \\
&&+(f(2,2)-f(2,1))+t(f(3,2)-f(2,1))\nonumber \\
&=&\frac{(n-1)(n-2)}{2}(f(3,3)-f(2,2))\nonumber \\
&\leq&0,\nonumber
\end{eqnarray}
where the last equality follows from Lemma \ref{lem-2-2}. The equality of the last inequality holds if and only if $n=2$, namely $G_0\cong K_2$.

Now we consider the second subcase when $d(G_0)=d_0\geq 2$. Since $\frac{d(G)}{2}\leq r(G)\leq d(G)$ it follows that $\lceil\frac{d_0}{2}\rceil\leq e_{G_0}(u)\leq d_0$ for any $u\in V(G_0)$. Notice that if $e_{G_0}(u)\geq d(u,u_0)+2$ then $e_{G_1}(u)=e_{G_2}(u)=e_{G_0}(u)$; if $e_{G_0}(u)\leq d(u,u_0)+1$ then $e_{G_1}(u)=e_{G_2}(u)+1=d(u,u_0)+2$. Furthermore, there is a vertex $u\in V(G_0)$ such that $d(u,u_0)=d_0=d(G_0)$, and so $e_{G_1}(u)=e_{G_2}(u)+1$. Hence,
\begin{eqnarray}
&&ABC_{3}(G_{1})-ABC_{3}(G_{2})\nonumber \\
&=&\sum\limits_{uv\in E(G_{0})}\left(f(e_{G_{1}}(u),e_{G_{1}}(v))-f(e_{G_{2}}(u),e_{G_{2}}(v))\right) \nonumber \\
& &+\sum\limits_{uv\in E(S_{t+2})}\left(f(e_{G_{1}}(u),e_{G_{1}}(v))-f(e_{G_{2}}(u),e_{G_{2}}(v))\right) \\
&<&\sum\limits_{uv\in E(S_{t+2})}\left(f(e_{G_{1}}(u),e_{G_{1}}(v))-f(e_{G_{2}}(u),e_{G_{2}}(v))\right) \\
&=&(f(d_0,d_0+1)-f(d_0,d_0+1))+t(f(d_0+1,d_0+2)-f(d_0,d_0+1)) \\
&<&0.
\end{eqnarray}

\textbf{Case 2.} $d_0=d(G_0)-1$. If $d_0=d(G_0)-1=1$, then every vertex of $V(G_0)$ is adjacent to $u_0$. And so, $e_{G_{1}}(u)=e_{G_{2}}(u)+1$ for any vertex $u\in V(G_0)$. By formula (1), we have
\begin{eqnarray}
&&ABC_{3}(G_{1})-ABC_{3}(G_{2})\nonumber \\
&\leq&\sum\limits_{uv\in E(S_{t+2})}\left(f(e_{G_{1}}(u),e_{G_{1}}(v))-f(e_{G_{2}}(u),e_{G_{2}}(v))\right) \nonumber \\
&=&(f(2,2)-f(1,2))+t(f(2,3)-f(1,2)) \nonumber \\
&=&0.\nonumber
\end{eqnarray}
The equality in last inequality holds if and only if $G_0\cong S_n$ with center $u_0$, since if $G_0$ is not such graph then it has an edge $xy$ not incident with $u_0$, and so $e_{G_1}(x)=e_{G_1}(y)=3$ but $e_{G_2}(x)=e_{G_2}(y)=2$, which implies formula (2).

In the second subcase when $d_0=d(G_0)-1\geq 2$, we have that $e_{G_1}(u_0)=e_{G_2}(u_0)=e_{G_0}(u_0)=d_0\geq 2$, and that $e_{G_1}(x)\geq e_{G_2}(x)\geq 2$ for every vertex $x\in V(G_0)-\{u_0\}$. Furthermore, for the vertex $u\in V(G_0)$ with $d(u_0,u)=d_0\geq 2$ we deduce that $e_{G_1}(u)=d_0+2\geq 4$ and $e_{G_1}(x)\geq 3$ for any neighbor of $u$. Hence, formula (2) follows by the last two results of Lemma \ref{lem-2-2}. Therefore, formulas (3) and (4) are also true in this subcase.

\textbf{Case 3.} $d_0\leq d(G_0)-2$. Since $d_0\geq 1$, it follows that $d_0\geq \lceil\frac{d(G_0)}{2}\rceil\geq 2$ in this case. And so, $e_{G_1}(x)\geq e_{G_2}(x)\geq 2$ for any vertex $x\in V(G_0)$. Furthermore, for the vertex $u\in V(G_0)$ with $d(u,u_0)=d_0$ it has a neighbor $v$ such that $e_{G_1}(u)\geq 4$, $e_{G_2}(u)\geq 3$, $e_{G_1}(v)\geq 3$ and $e_{G_2}(v)\geq 2$. Therefore, formula (2) follows by the last two results of Lemma \ref{lem-2-2}, which implies formulas (3) and (4).
\end{proof}

Let $H(n,k;n_{1},n_{2},\ldots, n_{k})$ be a unicyclic graph of order $n$ obtained from a cycle $C_{k}=v_{1}v_{2}\cdots v_{k}$ by attaching  $n_{i}$ new vertices to vertex $v_{i}$ for every $i=1,2,\ldots,k$, where $\sum\limits^{k}_{i=1}n_{i}=n-k$.

\begin{lemma}\label{lem-2-5}
Let $G\in U_n(k)$ be a unicyclic graph of order $n\geq 6$. If $k=3$ then $ABC_3(G)\leq n\sqrt{\frac{1}{2}}$, with equality holding if and only if $G\cong H(n,3;n_1,n_2,n_3)$; if $k\geq 4$ then
\begin{eqnarray}
&&ABC_3(G)  \nonumber \\
&\leq&
\left\{ \begin{array}{ll}
(n-k+2)f(\frac{k}{2}+1,\frac{k}{2})+(k-2)f(\frac{k}{2},\frac{k}{2})  \quad \mbox{if $k$ is even }; \\
(n-k+2)f(\frac{k+1}{2},\frac{k-1}{2})+(k-3)f(\frac{k-1}{2},\frac{k-1}{2})+f(\frac{k+1}{2},\frac{k+1}{2}) \quad \mbox{if $k$ is odd},
\end{array}
\right.
\end{eqnarray}
with equality holding if and only if $G\cong H(n,k; n-k,0,\ldots,0)$.
\end{lemma}
\begin{proof}
Let $G\in U_n(k)$ be a unicyclic graph with the maximum $ABC_3$ index. By Lemma \ref{lem-2-4}, $G$ must be of the form $H(n,k;n_1,n_2,\ldots,n_k)$. And so,
$$ABC_3(G)=\sum\limits_{v_iv_j\in E(C_k)}f(e_G(v_i),e_G(v_j))+\sum\limits_{\substack{v_ix\in E(G)\\ d(x)=1}}f(e_G(v_i),e_G(x)).$$

Firstly, we consider the case when $k=3$. If $n_1,n_2,n_3\geq 1$ then $e_{G}(v_i)=2$ for every vertex $v_i$ in the cycle and $e_G(x)=3$ for every pendent vertex of $G$. From the last two results of Lemma \ref{lem-2-2} it follows that
$$ABC_3(G)=3\sqrt{\frac{1}{2}}+\sum\limits^3_{i=1}k_if(2,3)=n\sqrt{\frac{1}{2}}.$$
Since $n\geq 6$ and $k=3$, it follows that at least one of $n_1,n_2$ and $n_3$ is not zero. And so, for every edges of $G$, at least one of its end has eccentricity two. Hence, the above formula is true in any case.

Secondly, we consider the case when $k\geq 4$. In this case, one can deduce easily that $e_G(x)\geq 2$ for every vertex $x$ of graph $G$. By Lemma \ref{lem-2-1} and \ref{lem-2-3}, to maximize $ABC_3(G)$ it suffices to maximize $f(e_G(u),e_G(u)+1)$ for every edge $uv$ with $e_G(v)=e_G(u)+1$, and to maximize $f(e_G(u),e_G(u))$  for every edge $uv$ with $e_G(v)=e_G(u)$. Let $C=v_1v_2\cdots v_kv_1$ be the cycle of graph $G$. Notice that $e_G(v_i)\leq e_C(v_i)+1$ for every vertex $v_i\in V(C)$, and that $e_G(x)=e_C(v_i)+1$ or $e_C(v_i)+2$ for every pendent vertex $x$ adjacent to vertex $v_i\in V(C)$, with $e_G(x)=e_C(v_i)+1$ if and only if $n_{i+\frac{k}{2}(\text{mod}\ k)}=0$ when $k$ is even, and $n_{i+\frac{k-1}{2}(\text{mod}\ k)}=n_{i+\frac{k+1}{2}(\text{mod}\ k)}=0$ when $k$ is odd. In conclusion, to maximize $ABC_3(G)$ if and only if to attach every pendent vertex of graph $G$ to a unique vertex of cycle $C$, and so the extremal graph is as prescribed in this lemma.

In the extremal graph $G$, only one vertex of $C$ has eccentricity $\frac{k}{2}+1$ when $k$ is even, and exactly two vertices of $C$ has eccentricity $\frac{k+1}{2}$ when $k$ is odd. Therefore, when $k$ is even
\begin{eqnarray}
ABC_3(G)&=&(n-k)f(\frac{k}{2}+1,\frac{k}{2})+2f(\frac{k}{2}+1,\frac{k}{2})+(k-2)f(\frac{k}{2},\frac{k}{2})  \nonumber \\
&=&(n-k+2)f(\frac{k}{2}+1,\frac{k}{2})+(k-2)f(\frac{k}{2},\frac{k}{2});\nonumber
\end{eqnarray}
when $k$ is odd,
\begin{eqnarray}
ABC_3(G)&=&(n-k)f(\frac{k+1}{2},\frac{k-1}{2})+f(\frac{k+1}{2},\frac{k+1}{2})\nonumber \\
&&+2f(\frac{k+1}{2},\frac{k-1}{2})+(k-3)f(\frac{k-1}{2},\frac{k-1}{2})  \nonumber \\
&=&(n-k+2)f(\frac{k+1}{2},\frac{k-1}{2})+(k-3)f(\frac{k-1}{2},\frac{k-1}{2}) \nonumber \\
&&+f(\frac{k+1}{2},\frac{k+1}{2}).\nonumber
\end{eqnarray}
And so, the lemma follows.
\end{proof}

\begin{theorem}\label{thm-2-6}
If $G$ is a unicyclic graph of order $n\geq 6$ then $ABC_3(G)\leq n\sqrt{\frac{1}{2}}$, with equality holding if and only if $G\in \{H(n,3;n_1,n_2,n_3),H(n,4;n-4,0,0,0)\}$.
\end{theorem}
\begin{proof}
If $G$ has even girth $k$ with $6\leq k\leq n-1$, let
$$f(k)=(n-k+2)f(\frac{k}{2}+1,\frac{k}{2})+(k-2)f(\frac{k}{2},\frac{k}{2}),$$
then
$$f(k-2)=(n-k+4)f(\frac{k}{2},\frac{k}{2}-1)+(k-4)f(\frac{k}{2}-1,\frac{k}{2}-1).$$
From Lemma \ref{lem-2-3} it follows that $f(k)<f(k-2)$. Hence, by the second results of Lemma \ref{lem-2-5} we have
\begin{eqnarray}
ABC_3(G)&\leq&ABC_3(H(n,k;n-k,0,\ldots,0))\nonumber\\
&<&ABC_3(H(n,k-2;n-k+2,0,\ldots,0))\nonumber\\
&<&\cdots\nonumber\\
&<&ABC_3(H(n,4;n-4,0,0,0))\nonumber\\
&=&n\sqrt{\frac{1}{2}}.\nonumber
\end{eqnarray}
If $k$ is odd with $5\leq k\leq n-1$, let
$$g(k)=(n-k+2)f(\frac{k+1}{2},\frac{k-1}{2})+(k-3)f(\frac{k-1}{2},\frac{k-1}{2})+f(\frac{k+1}{2},\frac{k+1}{2}),$$
then $g(k)$ has similar property as $f(k)$. And so, $g(k)<ABC_3(H(n,3;n_1,n_2,n_3))=n\sqrt{\frac{1}{2}}$. The theorem follows.
\end{proof}

\section{The second maximal $ABC_3$ index of unicyclic graphs}

For any integer $t\geq 1$ and $n_1\geq t+1$, let $H^t(n,3;n_1-t,n_2,n_3)$ be the unicyclic graph obtained from $H(n,3;n_1-t,n_2,n_3)$ by attaching $t$ new vertices to any given pendent vertex adjacent to $v_1$ of $H(n,3;n_1-t,n_2,n_3)$. Similarly, $H^t(n,k;n-k-t,0,\ldots,0)$ is the unicyclic graph obtained from $H(n-t,k;n-k-t,0,\ldots,0)$ by identifying one of its pendent vertex with the center of star $S_{t+1}$. This section will characterize the unicyclic graph with second maximal $ABC_3$ index.

\begin{lemma}\label{lem-3-1}
If integer $t\geq 1$ then $ABC_3(H^t(n,3;n-3-t,0,0))=(n-1)\sqrt{\frac{1}{2}}+f(3,3)$; if integers $t\geq 2$ and $k\geq 4$ then $ABC_3(H^t(n,k;n-k-t,0,\ldots,0))<ABC_3(H^1(n,k;n-k-1,0,\ldots,0))$.
\end{lemma}
\begin{proof}
Let $w$ be any pendent vertex of $H^t(n,k;n-k-t,0,\ldots,0)$ that has maximum eccentricity, and $u$ be its neighbor. Noticing that the cycle $C$ of $H^t(n,k;n-k-t,0,\ldots,0)$ has unique vertex $v$ with $e_{H^t(n,k;n-k-t,0,\ldots,0)}(v)=e_{H^t(n,k;n-k-t,0,\ldots,0)}(w)=\frac{k}{2}+2$ if $k$ is even, and that $C$ has exactly two vertices $v_i$ and $v_j$ with maximum eccentricity $\frac{k+3}{2}$ if $k$ is odd, when $k$ is even we have
\begin{eqnarray}
&&ABC_3(H^t(n,k;n-k-t,0,\ldots,0))\nonumber\\
&=&(t+2)f(\frac{k}{2}+2,\frac{k}{2}+1)+(n-k-t+2)f(\frac{k}{2}+1,\frac{k}{2})+(k-4)f(\frac{k}{2},\frac{k}{2}).\nonumber
\end{eqnarray}
Since $k\geq 4$, from Lemma \ref{lem-2-3} it follows that $f(\frac{k}{2}+1,\frac{k}{2})>f(\frac{k}{2}+2,\frac{k}{2}+1)$. And so,
\begin{eqnarray}
&&ABC_3(H^t(n,k;n-k-t,0,\ldots,0))-ABC_3(H^1(n,k;n-k-1,0,\ldots,0))\nonumber\\
&=&(t-1)(f(\frac{k}{2}+2,\frac{k}{2}+1)-f(\frac{k}{2}+1,\frac{k}{2}))\nonumber\\
&<&0.\nonumber
\end{eqnarray}
Similarly, when $k\geq 5$ is odd we have
\begin{eqnarray}
&&ABC_3(H^t(n,k;n-k-t,0,\ldots,0))\nonumber\\
&=&(t+2)f(\frac{k+3}{2},\frac{k+1}{2})+(n-k-t+2)f(\frac{k+1}{2},\frac{k-1}{2})\nonumber\\
&&+f(\frac{k+3}{2},\frac{k+3}{2})+(k-5)f(\frac{k-1}{2},\frac{k-1}{2}).\nonumber
\end{eqnarray}
And so, from Lemma \ref{lem-2-3} it follows that
\begin{eqnarray}
&&ABC_3(H^t(n,k;n-k-t,0,\ldots,0))-ABC_3(H^1(n,k;n-k-1,0,\ldots,0))\nonumber\\
&=&(t-1)(f(\frac{k+3}{2},\frac{k+1}{2})-f(\frac{k+1}{2},\frac{k-1}{2}))\nonumber\\
&<&0.\nonumber
\end{eqnarray}
Finally, when $k=3$, since $t\geq 1$ by direct calculation we have that $ABC_3(H^t(n,3;n-3-t,0,0))=(n-1)\sqrt{\frac{1}{2}}+f(3,3)$. And so, the lemma follows.
\end{proof}

\begin{lemma}\label{lem-3-2}
Let $G$ be an $n$-vertex unicyclic graph with $n\geq 6$ and girth $k\geq 3$. If $k=3$ and $G\ncong  H(n,3;n_1,n_2,n_3)$ then $ABC_3(G)\leq (n-1)\sqrt{\frac{1}{2}}+f(3,3)$, with equality holding if and only if $G\cong H^t(n,3;n-3-t,0,0)$; if $k=4$ and $G\ncong H(n,4;n-4,0,0,0)$ then $ABC_3(G)\leq (n-1)\sqrt{\frac{1}{2}}+f(3,3)$, with equality holding if and only if $G\cong H(n,4;n_1,n-4-n_1,0,0)$ with $n_1\geq 1$ and $n-n_1-4\geq 1$.
\end{lemma}
\begin{proof}
Let $G$ be an extremal unicyclic graph with girth $k$ and second maximal $ABC_3$. When $k=3$, since $G\ncong H(n,3;n_1,n_2,n_3)$ it follows from Lemma \ref{lem-2-4} that $G\cong H^t(n,3;n_1-t,n_2,n_3)$ with $t,n_1-t\geq 1$. If $n_2\geq 1$ or $n_3\geq 1$, noticing that $n_1+n_2+n_3=n$, by Lemma \ref{lem-2-3} we have
\begin{eqnarray}
&&ABC_3(H^t(n,3;n_1-t,n_2,n_3))\nonumber\\
&=&(n_2+n_3+t)f(3,4)+(n_1-t+2)f(2,3)+f(3,3)\nonumber\\
&<&(n-1)\sqrt{\frac{1}{2}}+f(3,3)\nonumber\\
&=&ABC_3(H^t(n,3;n-t-3,0,0)).\nonumber
\end{eqnarray}
This contradiction and Lemma \ref{lem-3-1} imply that $n_2=n_3=0$, namely $G\cong H^t(n,3;n-t-3,0,0)$ for any $1\leq t\leq n_1-1$. And so, the first statement of Lemma \ref{lem-3-2} is true.

When $k=4$, since $G\ncong H(n,4;n-4,0,0,0)$, from Lemma \ref{lem-2-4} it follows that $G\in \{H(n,4;n_1,\ldots,n_4),H^t(n,4;n-4-t,0,0,0)\}$, where $|\{i:n_i\geq 1,i=1,2,3,4\}|\geq 2$ and $t\geq 1$. By Lemma \ref{lem-3-1}, we deduce that $t=1$. Now, by direct calculation one can obtain easily that
\begin{eqnarray}
ABC_3(H^1(n,4;n-5,0,0,0))
&=&f(3,4)+(n-1)\sqrt{\frac{1}{2}}\nonumber\\
&<&f(3,3)+(n-1)\sqrt{\frac{1}{2}}\nonumber\\
&=&ABC_3(H(n,4;n_1,n-n_1-4,0,0)).\nonumber
\end{eqnarray}

Similarly, one can deduce by direct calculation that $ABC_3(H(n,4;n_1,n_2,n_3,n_4))\leq ABC_3(H(n,4;n_1,n-n_1-4,0,0))$ if $|\{i:n_i\geq 1,i=1,2,3,4\}|\geq 2$, with equality holding if and only if $H(n,4;n_1,n_2,n_3,n_4)$ is isomorphic to $H(n,4;n_1,n-n_1-4,0,0)$. Hence, the second statement of this lemma is also true.
\end{proof}

\begin{theorem}\label{thm-3-1}
If $G\notin \{H(n,4;n-4,0,0,0),H(n,3;n_1,n_2,n_3)\}$ is a unicyclic graph of order $n\geq 6$ then $ABC_3(G)\leq(n-1)\sqrt{\frac{1}{2}}+f(3,3)$, with equality holding if and only if $G\in \{H(n,5;n-5,0,0,0,0),H(n,4;n_1,n-n_1-4,0,0),H^t(n,3;n-t-3,0,0)\}$, where $n_1,n-n_1-4\geq 1$ for $H(n,4;n_1,n-n_1-4,0,0)$, and $t,n-t-3\geq 1$ for $H^t(n,3;n-t-3,0,0)$.
\end{theorem}
\begin{proof}
Let $G\notin \{H(n,4;n-4,0,0,0),H(n,3;n_1,n_2,n_3)\}$ be an extremal unicyclic graph with second maximal $ABC_3$ index. By Lemma \ref{lem-3-2}, the theorem is true when graph $G$ has girth $k=3$ or $4$. When $k\geq 5$, by Lemma \ref{lem-2-5} we have that $G\cong H(n,k;n-k,0,\ldots,0)$. If $k\geq 6$ is even, from Lemma \ref{lem-2-3} and \ref{lem-2-5} it follows that
\begin{eqnarray}
ABC_3(G)&=&(n-k+2)f(\frac{k}{2}+1,\frac{k}{2})+(k-2)f(\frac{k}{2},\frac{k}{2})\nonumber\\
&\leq&(n-k+2)f(4,3)+(k-2)f(3,3)\nonumber\\
&<&(n-1)\sqrt{\frac{1}{2}}+f(3,3)\nonumber\\
&=&ABC_3(H(n,4;n_1,n-n_1-4,0,0)), \nonumber
\end{eqnarray}
where $n_1,n-n_1-4\geq 1$. This contradiction shows that $G$ cannot has even girth $k\geq 6$.

If $G$ has odd girth $k\geq 5$, by Lemma \ref{lem-2-3} and \ref{lem-2-5} we deduce that
\begin{eqnarray}
&&ABC_3(G)\nonumber\\
&=&(n-k+2)f(\frac{k+1}{2},\frac{k-1}{2})+(k-3)f(\frac{k-1}{2},\frac{k-1}{2})+f(\frac{k+1}{2},\frac{k+1}{2})\nonumber\\
&\leq&(n-k+2)f(3,2)+(k-3)f(2,2)+f(3,3)\nonumber\\
&=&(n-1)\sqrt{\frac{1}{2}}+f(3,3),\nonumber
\end{eqnarray}
where the equality hold in the inequality if and only if $k=5$. Therefore, the unique extremal graph in this case is $H(n,5;n-5,0,0,0,0)$. And so, the theorem follows.
\end{proof}

\section{Maximal $ABC_3$ index of trees}
Let $C(a_1,a_2,\ldots,a_{d-1})$ be a caterpillar, which is obtained from path $P_d=v_0v_1\cdots v_d$ by attaching $a_i$ isolated vertices to vertex $v_i$ for every $i=1,2,\ldots,d-1$. Clearly, $C(a_1,a_2,\ldots,a_{d-1})$ has diameter $d$ and order $n=\sum\limits_{i=1}^{d-1}a_i+d+1$. For simplicity, let $C_{n,d}$ denote the set of graph $C(0,\ldots,0,a_{\frac{d}{2}},0,\ldots,0)$ when $d$ is even, and $C_{n,d}$ denote the set of all such graphs $C(0,\ldots,0,a_{\frac{d-1}{2}},a_{\frac{d+1}{2}},0,\ldots,0)$ when $d$ is odd.

\begin{lemma}\label{lem-4-1}
If $T$ is an $n$-vertex tree with diameter $d\geq 2$, then
\begin{eqnarray}\nonumber
ABC_3(T)
\leq
\left\{ \begin{array}{ll}
(n-1)\sqrt{\frac{1}{2}}\quad \mbox{$d=2,3$}; \\
(n-d-1)f(\frac{d}{2}+1,\frac{d}{2})+ABC_3(P_d)  \quad \mbox{if $d\geq 4$ is even}; \\
(n-d-1)f(\frac{d+3}{2},\frac{d+1}{2})+ABC_3(P_d) \quad \mbox{if $d\geq 5$ is odd},
\end{array}
\right.
\end{eqnarray}
with the equality holding if and only if $T\in C_{n,d}$.
\end{lemma}
\begin{proof}
The lemma is obviously true when $2\leq d\leq 3$, since $T$ is the prescribed caterpillar and its $ABC_3$ index can be obtained by direct calculation. So, we assume in what follows that $d\geq 4$, and $T$ has maximal $ABC_3$ index. Let $P_d=v_0v_1\cdots v_d$ be a longest path of $T$. By Lemma \ref{lem-2-4}, we deduce that $T\cong C(a_1,a_2\ldots, a_{d-1})$.

When $d$ is even, $e_{P_d}(v_r)>e_{P_d}(v_{\frac{d}{2}})=\frac{d}{2}\geq 2$ for $v_r\in V(P_d)\backslash v_{\frac{d}{2}}$. And so, by Lemma \ref{lem-2-3}, if we delete the $a_r$ pendent vertices adjacent to $v_r$ and add $a_r$ pendent vertices adjacent to the center vertex $v_{\frac{d}{2}}$, the $ABC_3$ index will increase. This observation shows that $T\in C_{n,d}$ in this case. Similarly, one can deduce with case that $T\in C_{n,d}$ when $d\geq 5$ is odd. Their $ABC_3$ index can also be easily obtained as are listed in the lemma.
\end{proof}

\begin{theorem}\label{thm-4-2}
If $T$ is an $n$-vertex tree with diameter $d\geq 2$ then $ABC_3(T)\leq (n-1)\sqrt{\frac{1}{2}}$, with equality holding if and only if $T\in C_{n,3}\cup \{S_n\}$. Furthermore, $C_{n,4}$ has second maximal $ABC_3$ index.
\end{theorem}
\begin{proof}
Firstly, by direct calculation, we have
\begin{eqnarray}\nonumber
ABC_3(P_d)=
\left\{ \begin{array}{ll}
2(f(\frac{d}{2}+1,\frac{d}{2})+f(\frac{d}{2}+2,\frac{d}{2}+1)+\cdots+f(d,d-1))  \quad \mbox{if $d\geq 4$ is even}; \\
f(\frac{d+1}{2},\frac{d+1}{2})+2(f(\frac{d+3}{2},\frac{d+1}{2})+f(\frac{d+5}{2},\frac{d+3}{2})+\cdots+f(d,d-1)) \quad \mbox{if $d\geq 5$ is odd}.
\end{array}
\right.
\end{eqnarray}
For simplicity, the following symbol $C_{n,d}$ denotes any graph in it. When $d\geq 4$ is even, we can deduce by Lemma \ref{lem-4-1} and Lemma \ref{lem-2-3} that
\begin{eqnarray}
&&ABC_3(C_{n,d+1})-ABC_3(C_{n,d}) \nonumber\\
&=&(n-d-2)f(\frac{d}{2}+2,\frac{d}{2}+1)+f(\frac{d}{2}+1,\frac{d}{2}+1)+2(f(\frac{d}{2}+2,\frac{d}{2}+1)\nonumber\\
&&+f(\frac{d}{2}+3,\frac{d}{2}+2)+\cdots+f(d+1,d))-(n-d-1)f(\frac{d}{2}+1,\frac{d}{2})\nonumber\\
&&-2(f(\frac{d}{2}+1,\frac{d}{2})+f(\frac{d}{2}+2,\frac{d}{2}+1)+\cdots+f(d,d-1))\nonumber\\
&<&0.\nonumber
\end{eqnarray}
Similarly, when $d\geq 5$ is odd, by Lemma \ref{lem-4-1} and Lemma \ref{lem-2-3} we have
\begin{eqnarray}
&&ABC_3(C_{n,d+1})-ABC_3(C_{n,d}) \nonumber\\
&=&(n-d-2)f(\frac{d+1}{2}+1,\frac{d+1}{2})+2(f(\frac{d+1}{2}+1,\frac{d+1}{2})
+f(\frac{d+1}{2}+2,\frac{d+1}{2}+1)\nonumber\\
&&+\cdots+f(d+1,d))-(n-d-1)f(\frac{d+3}{2},\frac{d+1}{2})-f(\frac{d+1}{2},\frac{d+1}{2})\nonumber\\
&&-2(f(\frac{d+3}{2},\frac{d+1}{2})+f(\frac{d+5}{2},\frac{d+3}{2})+\cdots+f(d,d-1))\nonumber\\
&=&(n-d-2)f(\frac{d+1}{2}+1,\frac{d+1}{2})+2f(d+1,d)\nonumber\\
&&-(n-d-1)f(\frac{d+3}{2},\frac{d+1}{2})-f(\frac{d+1}{2},\frac{d+1}{2})\nonumber\\
&<&0.\nonumber
\end{eqnarray}
These two observation show that if $d\geq 5$ then $ABC_3(C_{n,d})<ABC_3(C_{n,4})$. Since
\begin{eqnarray}
ABC_3(C_{n,4})&=&(n-4-1)f(\frac{4}{2}+1,\frac{4}{2})+ABC_3(P_4)\nonumber\\
&=&(n-5)f(2,2)+2f(4,3)+2f(3,2)\nonumber\\
&=&(n-3)\sqrt{\frac{1}{2}}+2f(4,3)\nonumber\\
&<&(n-1)\sqrt{\frac{1}{2}},\nonumber
\end{eqnarray}
it follows that the theorem is true.
\end{proof}

\section*{Acknowledgements}This work is supported by the National Natural Science Foundation of China (Grant No. 12061074), the Doctoral Research Foundation of Longdong University (Grant No. XYBYZK2306, XYBYZK2307), Science and Technology Plan Program of Gansu Province of China (No. 24JRRM005), Teacher Innovation Program of Gansu Province (2025B-216), Engineering Project of Longdong University (No. HXZK2437).

\end{document}